\documentclass[a4paper, 10pt]{amsart}

\usepackage{amsmath, amsfonts, amssymb, mathtools, url, graphicx, enumitem}
\usepackage[usenames, dvipsnames]{color}
\usepackage[colorlinks=true,
			raiselinks=true,
			linkcolor=MidnightBlue,
			citecolor=Mahogany,
			urlcolor=ForestGreen,
			pdfauthor=Debasish Chatterjee,
			pdftitle={Controlling stochastic linear systems over noisy communication channels},
			pdfkeywords={receding horizon control, noisy control channel},
			pdfsubject={Technical Report},
			plainpages=false]{hyperref}
\usepackage[square, sort&compress, authoryear]{natbib}
\usepackage[varg]{pxfonts}
\newtheoremstyle{dcstyle}{10pt}{10pt}{\slshape}{}{\bfseries}{.}{ }{}
\newtheoremstyle{nonum}{10pt}{10pt}{}{}{\itshape}{.}{ }{\thmname{#1}\thmnote{ (\mdseries #3)}}
\theoremstyle{dcstyle}
\newtheorem{theorem}{Theorem}[section]

\newtheorem{lemma}[theorem]{Lemma}
\newtheorem{proposition}[theorem]{Proposition}

\theoremstyle{definition}

\newtheorem{assumption}[theorem]{Assumption}

\newtheorem{prob}[theorem]{Problem}

\theoremstyle{remark}
\newtheorem{remark}[theorem]{Remark}
\newtheorem{example}[theorem]{Example}

\theoremstyle{nonum}

\DeclareMathOperator{\rank}{rank}
\DeclareMathOperator{\diam}{diam}

\DeclareMathOperator{\sat}{sat}
\DeclareMathOperator{\var}{var}

\DeclareMathOperator{\vecof}{vec}

\newcommand{\R}{\mathbb{R}}
\newcommand{\N}{\mathbb{N}}
\newcommand{\Nz}{\mathbb{N}_{0}}

\newcommand{\Let}{\coloneqq}

\newcommand{\acn}{\tilde u}
\newcommand{\uu}{\boldsymbol{u}}
\newcommand{\cnoise}{\nu}
\newcommand{\Cnoise}{\boldsymbol{\cnoise}}
\newcommand{\cnoiseset}{\mathrm{T}}
\newcommand{\Sprod}{\odot}
\newcommand{\EE}{\mathsf E}
\newcommand{\PP}{\mathsf P}
\newcommand{\sigalg}{\mathfrak F}
\newcommand{\Uset}{\mathcal U}
\newcommand{\Iset}{X}

\newcommand{\xz}{\bar x}
\newcommand{\Reach}{\mathfrak R}
\newcommand{\transp}{^{\mathsf T}}

\newcommand{\wt}{\widetilde}
\newcommand{\lra}{\longrightarrow}
\newcommand{\NSR}{\mathit{\Psi}}

\newcommand{\msbound}{\zeta}
\newcommand{\xSchur}{x^{(1)}}
\newcommand{\xOrtho}{x^{(2)}}

\renewcommand{\ge}{\geqslant}
\renewcommand{\le}{\leqslant}

\renewcommand{\leq}{\leqslant}

\newcommand{\abs}[1]{\left\lvert{#1}\right\rvert}
\newcommand{\norm}[1]{\left\lVert{#1}\right\rVert}
\newcommand{\Lp }[1]{\mathrm L_{#1}}
\newcommand{\secref}[1]{\S\ref{#1}}

\newcommand{\AssumptionEnd}{\hspace{\stretch{1}}{$\diamondsuit$}}
\newcommand{\ExampleEnd}{\hspace{\stretch{1}}{$\Delta$}}
\newcommand{\RemarkEnd}{\hspace{\stretch{1}}{$\triangleleft$}}

\allowdisplaybreaks

\title[Controlling stochastic linear systems over noisy control channels]{Mean-square boundedness of stochastic networked control systems with bounded control inputs}

\author[D.\ Chatterjee]{Debasish Chatterjee}
\author[S.\ Amin]{Saurabh Amin}
\author[P.\ Hokayem]{Peter Hokayem}
\author[J.\ Lygeros]{John Lygeros}
\author[S.\ S.\ Sastry]{S.\ Shankar Sastry}
\thanks{D. Chatterjee, P. Hokayem, J. Lygeros are with the Automatic Control Laboratory, ETH Z\"urich, Physikstrasse 3, 8092 Z\"urich, Switzerland. Their research was partially supported by the Swiss National Science Foundation under grant 200021-122072 and by the European Commission under the project Feednetback FP7-ICT-223866 (www.feednetback.eu).}%
\thanks{S.\ Amin is with the department of Civil and Environmental Engineering, University of California, Berkeley. His research was partially supported by the TRUST (Team for Research in Ubiquitous Secure Technology), which receives support from the National Science Foundation (NSF award number CCF-0424422).}%
\thanks{S.\ S.\ Sastry is with the department of Electrical Engineering and Computer Science, University of California, Berkeley, USA}%
\thanks{Emails: \{chatterjee,hokayem,lygeros\}@control.ee.ethz.ch, amins@berkeley.edu,$\qquad\qquad\qquad\qquad$ sastry@coe.berkeley.edu}%


\begin{document}

	\begin{abstract}
		We consider the problem of controlling marginally stable linear systems using bounded control inputs for networked control settings in which the communication channel between the remote controller and the system is unreliable. We assume that the states are perfectly observed, but the control inputs are transmitted over a noisy communication channel. Under mild hypotheses on the noise introduced by the control communication channel and large enough control authority, we construct a control policy that renders the state of the closed-loop system mean-square bounded.
	\end{abstract}

	\maketitle

	\section{Introduction}
		Communication channels have become ubiquitous in control applications such as remotely operated robotic systems \citep{ref:HokSpo-06}. In such applications, measurement and control signals are exchanged via a lossy and noisy communication channels, which makes the system a {\em networked control system} (NCS). The research in NCS has branched into many different directions that deal with the effects of delays, limited information exchanged, and information losses on the stability of the plant, see, e.g., \citep{ref:NaiFagZamEva-07} and the references therein.

		Control under information loss in the communication channel has been extensively studied within the Linear Quadratic Gaussian (LQG) framework \citep{ref:Imer,ref:SSFPS2007}. Typically, the communication channel(s) are modeled by an independent and identically distributed (i.i.d) Bernoulli process, which assign probabilities to the successful transmission of packets. Perhaps the most well known result in this setting is: When the transmission of sensor and control data packets happens over a network with TCP-like protocols, the closed-loop system under LQG controller can be mean-square stabilized provided that the probabilities of successful transmission are above a certain threshold. Since the TCP-like protocols enable the receiver to obtain an acknowledgment of whether or not the packets were successfully transmitted, the separation principle holds and the optimal LQG controller is linear in the estimated state. Thus, this result is a proper generalization of the classical LQG control problem to the networked control setting.

		Within the LQG setting, control inputs are not assumed bounded and therefore linear state feedback is a permissible and optimal strategy. However, guaranteeing  hard bounds on the control inputs is of paramount importance in applications.
Consequently, many researchers have pursued the problem of optimal control and stabilization for linear systems with bounded control inputs \citep{ref:Wonham, ref:Toivonen, ref:Saberi-99, ref:BernsteinMichel}. This problem has also received a renewed interest in recent years \citep{ref:RamChaMilHokLyg-09, ref:WangBoyd2009, ref:ChaHokLyg-09, ref:HokChaLyg-09, ref:Digailova}. In the deterministic setting, it is well-known \citep{ref:YanSonSus-97} that global asymptotic stabilization of a linear system $x_{t + 1} = A x_t + B u_t$ is possible if and only if the pair $(A, B)$ is stabilizable under arbitrary controls and the spectral radius of the system matrix $A$ is at most $1$. In the stochastic setting, it was argued in \citep{ref:NaiEva-04} that ensuring a mean-square bound for every initial condition is not possible for linear systems with bounded control elements if the system matrix $A$ is unstable. In \citep{ref:RamChaMilHokLyg-09} we established the existence of a policy with sufficiently large control authority that ensures mean-square boundedness of the states of the system under the assumption that $A$ is Lyapunov stable. Although Lyapunov stability of $A$ is a stronger requirement than the spectral radius of $A$ being at most $1$, to the best of our knowledge, this is the current state of the art. 

		In this article we generalize the results of \citep{ref:RamChaMilHokLyg-09} to incorporate noisy control channels. We consider mean-square boundedness of stochastic linear systems under the following specification:
		\begin{itemize}[label=$\circ$, leftmargin=*]
			\item the communication channel between the controller and the system actuators is noisy whereas the communication channel between sensors and controller is noiseless, and
			\item hard constraints on the control inputs must be satisfied.
		\end{itemize} 

		We are thus concerned with a networked setting as proposed in \citep{ref:Elia2005, ref:GSC2007, ref:SSFPS2007} when generalized to incorporate bounded control inputs~\citep{ref:RamChaMilHokLyg-09}. The control input $u^{(i)}$ for the $i$-th plant is communicated to the corresponding plant actuator via a lossy communication channel, which is characterized by the noise $\cnoise^{(i)}$ affecting the control input multiplicatively as shown in Figure \ref{fig:topology}.
		\begin{figure}[h]
			\label{fig:topology}
			\includegraphics[scale=0.85]{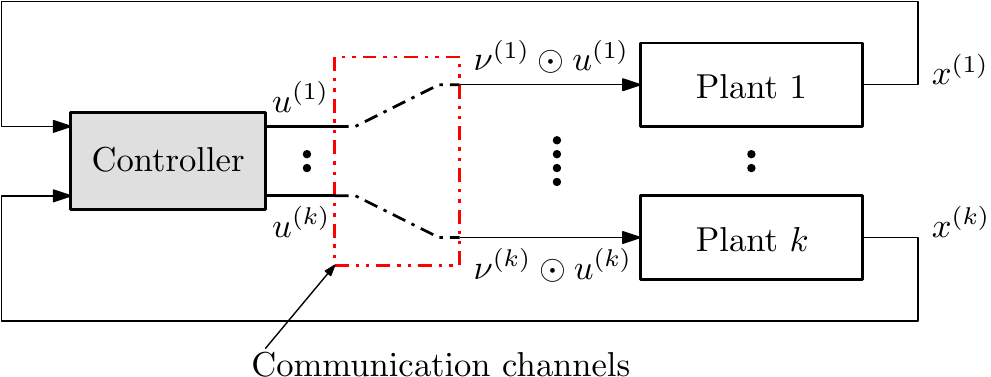}
			\caption{Topology of the control system.}
		\end{figure}
		We assume that the states are perfectly observed and are transmitted to the controller without any loss.

		\subsection*{Notation}
		For any random vector $\cnoise$ let $\mu_\cnoise \Let \EE[\cnoise]$ denote its mean and $\sigma_{\cnoise} \Let \var(\cnoise) \Let \EE\bigl[\norm{\cnoise - \mu_\cnoise}^2\bigr]$ denote its second moment. For a matrix $M$ we let $\norm M$ denote the induced Euclidean norm of $M$. We shall employ the standard notation $\diam(S) \Let \sup_{x, y\in S}\|x - y\|$ to denote the diameter of a subset $S$ of Euclidean space. For $n\in\N$, by $\mathbf{1}_{n}$ we denote a vector of length $n$ with all entries equal to $1$. For $r > 0$ and $n\in\N$, define the saturation function
		\begin{equation}
		\label{e:whatissatr}
			\R^n\ni z\longmapsto \sat_r(z) \Let
			\begin{cases}
				z & \text{if $\norm{z} < r$}\\
				rz/\norm{z} & \text{otherwise}
			\end{cases}
			\quad\in\R^n.
		\end{equation}

		The remainder of this article is organized as follows. In \secref{s:problem} we formalize the main problem with all the underlying assumptions, and in \secref{s:main} we state the main results. The proofs are provided in \secref{s:proofs}.

	\section{Problem Setup}
	\label{s:problem}
		Consider the following discrete-time stochastic linear system subjected to packet drops in the control communication channel
		\begin{equation}
		\label{e:sys}
		\begin{aligned}
			x_{t+1} & = A x_t + B \acn_t + w_t,\\
			\acn_t & \Let \cnoise_t \Sprod u_t,
		\end{aligned}
		\qquad \quad t\in\Nz,
		\end{equation}
		where $x_t\in\R^d$ is the state, $u_{t}\in\R^m$ is the control input, $A\in\R^{d\times d}$ is the dynamics matrix, $B\in\R^{d\times m}$ is the input matrix, $(w_t)_{t\in\Nz}$ is an $\R^d$-valued random process noise, and $(\cnoise_t)_{t\in\Nz}$ is an $\R^m$-valued random process modelling the uncertainty in the control communication channel, and $\Sprod$ denotes the Schur or Hadamard product of matrices.\footnote{Recall \cite[p. 444]{ref:Ber-09} that if $M', M''$ are $n_1\times n_2$ matrices with real entries, then $M'\Sprod M''$ is the $n_1\times n_2$ matrix defined by $(M'\Sprod M'')_{i, j} \Let (M')_{i, j} (M'')_{i, j}$.} The initial condition $x_0 = \xz$ is given and  the state $x_{t}$ is perfectly observed by the controller.
		
		The controller determines the control input $u_{t}$ based on the history of $k$ states $\Iset_{t,k} \Let (x_{t-k+1}, \ldots, x_{t-1},x_{t})$. (For $t = 0, \ldots, k-2$, $\Iset_{t, k} \Let (\underset{(k-1-t)\text{-times}}{\underbrace{x_0, \ldots, x_0}}, x_0, x_1, \ldots, x_t)$.) The controller synthesizes a deterministic control policy $\pi = (\pi_t)_{t\in\Nz}$ which maps the states vector $\Iset$ into a control set $\Uset$. To wit,
		\[
			u_t = \pi_{t}(\Iset_{t,k}), \quad t\in\Nz,
		\]
		where the maps $\pi_{t}:\R^{kd}\lra\Uset\subseteq\R^m$, $t\in\Nz$, are Borel measurable. Such a control policy $\pi$ is known as a \emph{$k$-history dependent policy}. The control set $\Uset$ is assumed to be nonempty, compact, and containing the origin. Any control policy $\pi = (\pi_{t})_{t\in\Nz}$ which guarantees that the control input sequence $(u_t)_{t\in\Nz}$ satisfies
		\begin{align}
		\label{e:uinUlabel}
			u_t\in\Uset, \qquad t\in\Nz,
		\end{align}
		is called an \emph{admissible $k$-history dependent policy}. 
In many practical situations involving saturating actuators and hard bounds on control inputs, $\Uset$ is chosen to be a ball, i.e.,
		\begin{align}
		\label{e:boxsetdef}
			\Uset \Let \bigl\{z\in\R^m\,\big|\, \norm{z} \le U_{\max}\bigr\},
		\end{align}
		where $U_{\max}>0$ is called the \emph{control authority} available to the controller.

		Our control objective is to synthesize an admissible $k$-history dependent policy which ensures that the second moment of the closed-loop system, for any initial condition $\xz\in \R^n$,
		\begin{align}
		\label{e:chclosedloop}
			x_{t+1}=Ax_{t}+B \cnoise_t \Sprod \pi_{t}(\Iset_{t,k})+w_{t}, \qquad t\in\Nz,
		\end{align}	
		remains bounded for all $t\in\Nz$. We shall focus on the following problem:
		\begin{prob}
		\label{prob:main}
			Find, if possible, a control authority $U_{\max}$ and an admissible policy $\pi=(\pi_t)_{t\in\Nz}$ with control authority $U_{\max}$, such that the following condition holds:
			\begin{quote}
				for every initial condition $x_0 = \xz\in\R^d$ there exists a constant $\msbound > 0$ such that the closed-loop system \eqref{e:chclosedloop} satisfies $\EE_{\xz}[\norm{x_t}^2] \le \msbound$ for all $t\in\Nz$.
			\end{quote}
		\end{prob}
		In practice, a performance index that accounts for the average sum of cost-per-stage functions (involving the state and control inputs) of the system is often required to be minimized; however, in this article we are only concerned with the stability property defined in Problem \ref{prob:main}.

		We shall make the following standing hypotheses:
		\begin{assumption}
		\label{a:basic}
			\mbox{}
			\begin{enumerate}[label={(\roman*)}, leftmargin=*, align=right, widest = iii]
				\item The matrix $A$ is Lyapunov stable, i.e., all the eigenvalues of $A$ lie in the closed unit circle, and all eigenvalues $\lambda$ satisfying $\abs{\lambda} = 1$ have equal algebraic and geometric multiplicities.\label{a:basic:stability}
				\item The pair $(A, B)$ is stabilizable.\label{a:basic:stabilizability}
				\item The process noise $(w_t)_{t\in\Nz}$ is an independent sequence, and has bounded fourth moment, i.e., $C_4 \Let \sup_{t\in\Nz} \EE[\norm{w_t}^4] < \infty$.\label{a:basic:4moment}
				\item The control input sequence~$(u_{t})_{t\in\Nz}$ satisfies~\eqref{e:uinUlabel} and~\eqref{e:boxsetdef}.\label{a:basic:uconstr}
				\item The control channel noise $(\cnoise_t)_{t\in\Nz}$ is i.i.d.\AssumptionEnd
			\end{enumerate}
		\end{assumption}

		It follows from Assumption \ref{a:basic}-\ref{a:basic:4moment} that there exists $C_1 > 0$ such that $\EE[\norm{w_t}] \le C_1$ for all $t\in\Nz$. (For instance, Jensen's inequality shows that $C_1 \le \sqrt[4]{C_4}$.) Note also that Assumption \ref{a:basic}-\ref{a:basic:4moment} does not require that the process noise vectors $(w_t)_{t\in\Nz}$ be identically distributed. The assumption of mutual independence of $(w_t)_{t\in\Nz}$ can also be relaxed, but we shall not pursue this line of generalization here.

		Without any loss of generality, we also assume that $A$ is in real Jordan canonical form (cf.\ \citep{ref:NaiEva-04}). Indeed, given a linear system described by system matrices $(\tilde A, \tilde B)$, there exists a coordinate transformation in the state-space that brings the pair $(\tilde A, \tilde B)$ to the pair $(A, B)$, where $A$ is in real Jordan form \cite[p. 150]{ref:HorJoh-90}. In particular, choosing a suitable ordering of the Jordan blocks, we can ensure that the pair $(A,  B)$ has the form $\left(\begin{bmatrix}A_{1} & 0\\ 0 & A_{2}\end{bmatrix}, \begin{bmatrix}B_1\\ B_2\end{bmatrix}\right)$, where $A_{1}\in\R^{d_1\times d_1}$ is Schur stable, and $A_{2}\in\R^{d_2\times d_2}$ has its eigenvalues on the unit circle. By Assumption \ref{a:basic}-\ref{a:basic:stability}, $A_{2}$ is therefore block-diagonal with elements on the diagonal being either $\pm 1$ or $2\times 2$ rotation matrices. As a consequence, $A_{2}$ is orthogonal.  Moreover, since $(A, B)$ is stabilizable by Assumption \ref{a:basic}-\ref{a:basic:stabilizability}, the pair $(A_{2}, B_2)$ must be reachable in a number of steps $\kappa \leq d_2$ that depends on the dimension of $A_{2}$ and the structure of $(A_{2}, B_2)$, i.e., $\rank(\Reach_\kappa(A_2,B_2)) = d_2$, where
		\[
			\Reach_\kappa(A_2,B_2) \Let \begin{bmatrix} A_2^{\kappa-1} B_2 & \cdots & A_2B_2 & B_2 \end{bmatrix}.
		\]
		The smallest such $\kappa$ is called the \emph{controllability index} of $(A_2,B_2)$ and is fixed throughout the rest of this article. Summing up, we can start by considering that the state equation  \eqref{e:sys} has the form
		\begin{equation}
		\label{e:Jordan}
			\begin{bmatrix} \xSchur_{t+1}\\  \xOrtho_{t+1}\end{bmatrix}
			= \begin{bmatrix} A_{1} \xSchur_t\\  A_{2} \xOrtho_t\end{bmatrix}
			+ \begin{bmatrix} B_1\\  B_2\end{bmatrix} \acn_t
			+ \begin{bmatrix} w^{(1)}_t\\ w^{(2)}_t\end{bmatrix},
		\end{equation}
		where $A_{1}$ is Schur stable, $A_{2}$ is orthogonal, and the subsystem $(A_{2}, B_2)$ is reachable in $\kappa$ steps. Since the matrix $\Reach_\kappa(A_2, B_2)$ has rank full rank, its Moore-Penrose pseudoinverse exists and is given by
		\[
		    \Reach_\kappa(A_2, B_2)^+ \Let \Reach_\kappa(A_2, B_2)\transp\bigl(\Reach_\kappa(A_2, B_2)\Reach_\kappa(A_2, B_2)\transp\bigr)^{-1}.
		\]

	\section{Main Results}
	\label{s:main}
		We are ready to state the main result pertaining to the existence of policy of bounded authority that renders the state of the system \eqref{e:sys} mean-square bounded. Let us define the normalized measure of dispersion or the noise-to-signal ratio of the channel $\NSR \Let \sqrt{{\sigma_\cnoise}}\cdot \max_{i=1, \ldots, m}\abs{({\mu_\cnoise})_i}^{-1}$. We impose the following additional requirements:

		\begin{assumption}
		\label{a:key}
			In addition to Assumption \ref{a:basic} we stipulate that:
			\begin{enumerate}[label={(\roman*)}, leftmargin=*, align=right, widest=vii, start=6]
				\item The control channel noise has bounded range, i.e., $\cnoise_0\in \cnoiseset$, where $\cnoiseset$ is a bounded subset of $\R^m$, and that $\mu_\cnoise$ has nonzero entries.\label{a:key:cnoisesetandnonzeroentries}
				\item The following two technical conditions hold:\label{a:key:techcond}
				\begin{enumerate}[label=(vii.\alph*), leftmargin=*, align=right]
					\item $\displaystyle{\kappa \NSR \norm{\Reach_\kappa(A_2, B_2)^+}\norm{\Reach_\kappa(A_2, B_2)} < 1}$.\label{a:key:meanvarrelation}
					\item $\displaystyle{U_{\max} > \frac{ \sqrt\kappa C_1 \left(\max_{i=1, \ldots, m}\abs{({\mu_\cnoise})_i}^{-1}\right) \norm{\Reach_\kappa(A_2, B_2)^+}\norm{\Reach_\kappa(A_2, I_2)} }{ 1 - \kappa  \NSR \norm{\Reach_\kappa(A_2, B_2)^+}\norm{\Reach_\kappa(A_2, B_2)} }}$.\label{a:key:Umaxbound}\AssumptionEnd
				\end{enumerate}
			\end{enumerate}
		\end{assumption}

		\begin{proposition}
		\label{p:main}
			Consider the system \eqref{e:sys}, and suppose that Assumption \ref{a:key} holds. Then there exists a $\kappa$-history dependent policy $\pi \Let (\pi_t)_{t\in\Nz}$ with control authority at most $U_{\max}$ (see \eqref{e:finalcontrolpolicy} below), such that for every initial condition $\xz$ there exists a constant $\msbound = \msbound(\xz, \kappa, \mu_\cnoise, \NSR, C_1) > 0$ with
			\[
				\EE_{\xz}[\norm{x_t}^2] \le \msbound\qquad \text{for all }t\in\Nz
			\]
			in closed-loop.
		\end{proposition}

		\begin{remark}
		\label{r:cnoiseset}
			 Proposition \ref{p:main} assumes minimal structure from the set in which the control channel noise takes its values. In particular, we do not assume that the control channel noise takes values in a finite set---in fact, $\cnoiseset$ may be uncountable. (While the standard choice of modelling uncertainty in the control channels has focussed on a multiplicative Bernoulli $\{0, 1\}$ random variable multiplying the entire control vector, there are cases in which the uncertainty model considered in \eqref{e:sys} (i.e., different random variables multiplying the components of the controller,) makes sense. For instance, the standard processes of control quantization or ``binning'' can be viewed as introducing uncertainty to the controller---components of the controller being multiplied by bounded but not necessarily identically distributed random variables; the set $\cnoiseset$ has the natural interpretation of the ``largest bin.'') In view of this, Assumption \ref{a:key}-\ref{a:key:meanvarrelation} is a technical condition stipulated as a trade-off for the absence of any further structure in the set $\cnoiseset$.\RemarkEnd
		\end{remark}

		In \secref{s:proofs} we prove Proposition \ref{p:main} by a constructive method. It turns out that our policy (see \eqref{e:finalcontrolpolicy} below) is derived from the $\kappa$-subsampled system $(x_{\kappa t})_{t\in\Nz}$, and is $\kappa$-history dependent. To wit, for each $n\in\Nz$, at time $\kappa n$, based on the state $x_{\kappa n}$, the policy synthesizes a $\kappa$-long sequence of control values for time steps $\kappa n, \kappa n + 1, \ldots, \kappa(n+1)-1$.

		Let us assume that the same uncertainty enters all the control channels, i.e., $\acn_t = \cnoise_t u_t$, where $\cnoise_t\in\R$. The structure of our control policy permits us to transmit the control data packets in a single burst each $\kappa$ steps. This however, necessitates the presence of a buffer at the actuator end of the plant to store the $\kappa$ control values $\{\cnoise_{\kappa n} u_{\kappa n}, \cnoise_{\kappa n} u_{\kappa n+1}, \ldots, \cnoise_{\kappa n}u_{\kappa (n+1)-1}\}$ transferred in a burst at time $\kappa n$, such that at each time $t\in \{\kappa n, \ldots, \kappa(n+1)-1\}$, the control $\cnoise_{\kappa n} u_t$ can be applied.

		\begin{assumption}
    	\label{a:burst}
		    In addition to Assumption \ref{a:basic}, we assume that:
			\begin{enumerate}[label=(\roman*$'$), leftmargin=*, align=right, widest=vii, start=6]
				\item Control signals are sent to the actuator every $\kappa$ steps, and for each $t\in\Nz$, the control channel noise is of the form $\cnoise_{\kappa t} \mathbf{1}_{\kappa m}$, with $\cnoise_{\kappa n}\in\{0,1\}$ and $\PP(\cnoise_{\kappa n} = 1) = p\in\;]0, 1[$ for each $n\in\Nz$.\label{a:burst:nonzeroentriesinmean}
				\item $U_{\max} > \sqrt\kappa C_1\norm{\Reach_\kappa(A_2, B_2)^+}\norm{\Reach_\kappa(A_2, I_2)}/p$.\label{a:burst:Umaxbound}\AssumptionEnd
			\end{enumerate}
		\end{assumption}

		\begin{proposition}
		\label{p:burst}
			Consider the system \eqref{e:sys}, and suppose that Assumption \ref{a:burst} holds. Then there exists a $\kappa$-history dependent policy $\pi \Let (\pi_t)_{t\in\Nz}$ with control authority at most $U_{\max}$ (see \eqref{e:burstpolicy} below), such that for every initial condition $\xz$ there exists a constant $\msbound' = \msbound'(\xz, \kappa, p, C_1) > 0$ with
			\[
				\EE_{\xz}[\norm{x_t}^2] \le \msbound'\qquad \text{for all }t\in\Nz
			\]
			in closed-loop.
		\end{proposition}

		\begin{remark}
			We noted in Remark \ref{r:cnoiseset} that Proposition \ref{p:main} assumes minimal structure from the bounded set $\cnoiseset$. In contrast, Proposition \ref{p:burst} assumes a rather specific structure of the set $\cnoiseset$---that it consists of two elements (note that $\cnoise_{\kappa n}\in\{0, 1\}$ for each $n\in\Nz$ in Assumption \ref{a:burst}-\ref{a:burst:nonzeroentriesinmean}). The i.i.d Bernoulli assumption on $(\cnoise_{\kappa t})_{t\in\Nz}$ leads to a simpler description of the control authority $U_{\max}$ in Assumption \ref{a:burst}-\ref{a:burst:Umaxbound} compared to Assumption \ref{a:key}-\ref{a:key:Umaxbound}, and the analog of Assumption \ref{a:key}-\ref{a:key:meanvarrelation} is not required here.\RemarkEnd
		\end{remark}

		As promised in Remark \ref{r:cnoiseset}, we provide a simple scalar example illustrating some effects of varying the probability of transmission of the control signal.
		\begin{example}
		\label{ex:scalar}
			Consider the scalar system $x_{t+1} = x_t + \acn_t + w_t$, $t\ge 0$, with initial condition $x_0 = \xz$, $\acn_t = \cnoise_t u_t$. Suppose that $(\cnoise_t)_{t\in\Nz}$ is i.i.d Bernoulli $\{0, 1\}$ with $\PP(\cnoise_t = 1) = p > 0$, and let $(w_t)_{t\in\Nz}$ be i.i.d, and satisfy $\sup_{t\in\Nz}\EE[\abs{w_t}^4] = C_4' < \infty$. This implies, in particular, that $\sup_{t\in\Nz}\EE[\abs{w_t}] \le C_1' \le \sqrt[4]{C_4'}$. Suppose that $U_{\max} > C_1'/p$, where $u_t\in[-U_{\max}, U_{\max}]$ for all $t$. With this much data it is easy to verify the conditions of Assumption \ref{a:burst}. We conclude by Proposition \ref{p:burst} that there exists a policy with control authority at most $U_{\max}$ such that the system is mean-square bounded. In fact, we see that for every nonzero probability $p$ of transmission of the control signal, there exists a control authority $U_{\max} > 0$ and a policy with control authority at most $U_{\max}$, under which the state of the system is mean-square bounded.\ExampleEnd
		\end{example}

		\begin{remark}
			Notice that Proposition \ref{p:burst} does not contradict the main results of \citep{ref:SSFPS2007}, where it was proved (see \citep[Lemma 5.4]{ref:SSFPS2007}) that there exists a threshold probability of i.i.d.\ Bernoulli packet drops such that a stabilizing linear feedback for unstable linear systems can be found provided the drop probability is less than that threshold. Indeed, in Assumption \ref{a:basic} we have specifically ruled out unstable $A$.\RemarkEnd
		\end{remark}

	\section{Proofs of Propositions \ref{p:main} and \ref{p:burst}}
	\label{s:proofs}

		For our proofs of Propositions \ref{p:main} and \ref{p:burst} we shall employ the following immediate adaptation of \citep[Theorem 1]{ref:PemRos-99} on $\Lp 2$ bounds of nonnegative random variables:

		\begin{proposition}[{\citep{ref:PemRos-99}}]
		\label{p:PR}
			Let $(\Omega, \sigalg, \PP)$ be a probability space, and let $(\sigalg_t)_{t\in\Nz}$ be a filtration on $(\Omega, \sigalg, \PP)$. Suppose that $(\xi_t)_{t\in\Nz}$ is a family of nonnegative random variables adapted to $(\sigalg_t)_{t\in\Nz}$, such that there exist constants $a, M, J > 0$ such that $\xi_0 < J$, and for all $t\in\Nz$,
			\begin{gather}
				\EE^{\sigalg_t}[\xi_{t+1} - \xi_t] \le -a \qquad \text{on the set }\{\xi_t > J\},\label{e:PR1}\\
				\EE[\abs{\xi_{t+1} - \xi_t}^4\mid \xi_0, \ldots, \xi_t] \le M.\label{e:PR2}
			\end{gather}
			Then there exists $c = c(a, J, M) > 0$ such that $\displaystyle{\sup_{t\in\Nz}\EE\bigl[\xi_t^2\bigr] \le c}$.
		\end{proposition}

		In what follows we let $I_2$ denote the $d_2\times d_2$ identity matrix.
		\begin{lemma}
		\label{l:PRcondition1}
			Given the system \eqref{e:sys}, suppose that Assumption \ref{a:basic} holds, and consider the decomposition \eqref{e:Jordan}. Let $\sigalg_t$ be the $\sigma$-algebra generated by $(x_t)_{t\in\Nz}$, i.e., $\sigalg_t \Let \sigma(x_n;\, n = 0, \ldots, t)$ for each $t\in\Nz$. Then there exists $a, J > 0$ such that
			\[
				\EE^{\sigalg_{\kappa t}}\Bigl[\norm{\xOrtho_{\kappa(t+1)}} - \norm{\xOrtho_{\kappa t}}\Bigr] \le -a \qquad\text{on the set }\Bigl\{\norm{\xOrtho_{\kappa t}} > J\Bigr\}\quad \text{for all }t\in\Nz.
			\]
		\end{lemma}
		\begin{proof}
			To simplify notation we first write compactly 
			\begin{equation}
			\label{e:whatisuu}
				\Cnoise_{\kappa t} \Let \begin{bmatrix}\cnoise_{\kappa t}\\\cnoise_{\kappa t + 1}\\\vdots\\\cnoise_{\kappa(t+1)-1}\end{bmatrix} \quad \text{and}\quad \uu_{\kappa t} \Let \begin{bmatrix}u_{\kappa t}\\u_{\kappa t + 1}\\\vdots\\u_{\kappa(t+1)-1}\end{bmatrix}.
			\end{equation}
			It follows from the system dynamics that
			\[
				\xOrtho_{\kappa(t+1)} = A_2^\kappa \xOrtho_{\kappa t} + \Reach_\kappa(A_2, B_2)\uu_{\kappa t} + \Reach_\kappa(A_2, I_2)w^{(2)}_{\kappa t:\kappa(t+1)-1},\qquad t\in\Nz,
			\]
			where $w^{(2)}_{\kappa t:\kappa(t+1)-1} \Let \begin{bmatrix} \bigl(w^{(2)}_{\kappa t}\bigr)\transp & \cdots & \bigl(w^{(2)}_{\kappa (t+1) - 1}\bigr)\transp\end{bmatrix}\transp$. Therefore,
			\begin{align*}
				& \EE^{\sigalg_{\kappa t}}\Bigl[\norm{\xOrtho_{\kappa(t+1)}} - \norm{\xOrtho_{\kappa t}}\Bigr]\\
				& = \EE^{\sigalg_{\kappa t}}\Bigl[\norm{A_2^\kappa \xOrtho_{\kappa t} + \Reach_\kappa(A_2, B_2)\uu_{\kappa t} + \Reach_\kappa(A_2, I_2)w^{(2)}_{\kappa t:\kappa(t+1)-1}} - \norm{\xOrtho_{\kappa t}}\Bigr]\\
				& \le \EE^{\sigalg_{\kappa t}}\Bigl[\norm{A_2^\kappa \xOrtho_{\kappa t} + \Reach_\kappa(A_2, B_2)\uu_{\kappa t}} - \norm{\xOrtho_{\kappa t}}\Bigr] + \norm{\Reach_\kappa(A_2, I_2)}\EE\Bigl[\norm{w^{(2)}_{\kappa t:\kappa(t+1)-1}}\Bigr]\\
				& \le \EE^{\sigalg_{\kappa t}}\Bigl[\norm{A_2^\kappa \xOrtho_{\kappa t} + \Reach_\kappa(A_2, B_2)\uu_{\kappa t}} - \norm{\xOrtho_{\kappa t}}\Bigr] + \sqrt\kappa \norm{\Reach_\kappa(A_2, I_2)} C_1.
			\end{align*}
			Since $A_2$ is orthogonal, we have $\norm{A_2^\kappa \xOrtho_{\kappa t}} = \norm{\xOrtho_{\kappa t}}$. We require $\uu_{\kappa t}$ be $\sigalg_{\kappa t}$-measurable. Employing Jensen's inequality and sublinearity of the square-root function, we get
			\begin{align*}
				\EE^{\sigalg_{\kappa t}} & \Bigl[\norm{A_2^\kappa \xOrtho_{\kappa t} + \Reach_\kappa(A_2, B_2)\uu_{\kappa t}}\Bigr]\\
				& \le \sqrt{ \EE^{\sigalg_{\kappa t}}\Bigl[\norm{A_2^\kappa \xOrtho_{\kappa t} + \Reach_\kappa(A_2, B_2)\uu_{\kappa t}}^2\Bigr] }\\
				& = \left(\norm{\xOrtho_{\kappa t}}^2 + 2\bigl(\xOrtho_{\kappa t}\bigr)\transp (A_2^\kappa)\transp\Reach_k(A_2, B_2)(\EE^{\sigalg_{\kappa t}}[\Cnoise_{\kappa t}]\Sprod\uu_{\kappa t})\right.\\
				& \qquad\qquad + \left.\EE^{\sigalg_{\kappa t}}\bigl[\bigl(\Reach_\kappa(A_2, B_2)(\Cnoise_{\kappa t}\Sprod\uu_{\kappa t})\bigr)\transp\bigl(\Reach_\kappa(A_2, B_2)(\Cnoise_{\kappa t}\Sprod\uu_{\kappa t})\bigr)\bigr]\right)^{1/2}\\
				& \le \norm{A_2^\kappa \xOrtho_{\kappa t} + \Reach_\kappa(A_2, B_2)(\EE^{\sigalg_{\kappa t}}[\Cnoise_{\kappa t}]\Sprod\uu_{\kappa t})}\\
				& \qquad\qquad + \sqrt{\EE^{\sigalg_{\kappa t}}\bigl[\norm{\Reach_\kappa(A_2, B_2)(\Cnoise_{\kappa t}\Sprod\uu_{\kappa t})}^2\bigr] - \norm{\Reach_\kappa(A_2, B_2)(\EE^{\sigalg_{\kappa t}}[\Cnoise_{\kappa t}]\Sprod\uu_{\kappa t})}^2},
			\end{align*}
			The last term under the square-root is simply the conditional variance of the vector $\Reach_\kappa(A_2, B_2)(\Cnoise_{\kappa t}\Sprod\uu_{\kappa t})$ given $\sigalg_{\kappa t}$. Since $(\cnoise_n)_{n=\kappa t}^{\kappa(t+1)-1}$ is independent of $\sigalg_{\kappa t}$, $\bar\Cnoise \Let \EE^{\sigalg_{\kappa t}}[\Cnoise_{\kappa t}]$ is a constant, and equals $\vecof\{\underset{\kappa\text{-times}}{\underbrace{{\mu_\cnoise}, \ldots, {\mu_\cnoise}}}\}$.) Thus, we see that
			\begin{align*}
				& \EE^{\sigalg_{\kappa t}}\Bigl[\norm{A_2^\kappa \xOrtho_{\kappa t} + \Reach_\kappa(A_2, B_2)\uu_{\kappa t}}\Bigr]\\
				& \;\le \norm{A_2^\kappa \xOrtho_{\kappa t} + \Reach_\kappa(A_2, B_2)(\bar\Cnoise\Sprod\uu_{\kappa t})} + \sqrt{\EE^{\sigalg_{\kappa t}}\bigl[\norm{\Reach_\kappa(A_2, B_2)((\Cnoise_{\kappa t} - \bar\Cnoise)\Sprod\uu_{\kappa t})}^2\bigr]}\\
				& \;\le \norm{A_2^\kappa \xOrtho_{\kappa t} + \Reach_\kappa(A_2, B_2)(\bar\Cnoise\Sprod\uu_{\kappa t})} + \sqrt\kappa\norm{\Reach_\kappa(A_2, B_2)}U_{\max}\sqrt{\EE^{\sigalg_{\kappa t}}\bigl[\norm{\Cnoise_{\kappa t} - \bar\Cnoise}^2\bigr]}\\
				& \;\le \norm{A_2^\kappa \xOrtho_{\kappa t} + \Reach_\kappa(A_2, B_2)(\bar\Cnoise\Sprod\uu_{\kappa t})} + \kappa\norm{\Reach_\kappa(A_2, B_2)}U_{\max}\sqrt{{\sigma_\cnoise}}.
			\end{align*}
			Collecting the inequalities above, we see that
			\begin{equation}
			\label{e:semifinalineq}
			\begin{aligned}
				\EE^{\sigalg_{\kappa t}}\Bigl[\norm{\xOrtho_{\kappa(t+1)}} - & \norm{\xOrtho_{\kappa t}}\Bigr] \le \norm{A_2^\kappa \xOrtho_{\kappa t} + \Reach_\kappa(A_2, B_2)(\bar\Cnoise\Sprod\uu_{\kappa t})} - \norm{A_2^\kappa \xOrtho_{\kappa t}}\\
				& \qquad + \kappa\bigl(\norm{\Reach_\kappa(A_2, B_2)}U_{\max}\sqrt{{\sigma_\cnoise}} + \norm{\Reach_\kappa(A_2, I_2)} C_1/\sqrt{\kappa}\bigr).
			\end{aligned}
			\end{equation}
			In view of Assumption \ref{a:key}-\ref{a:key:meanvarrelation} we see that there exists $ 0 < a \Let U_{\max} \bigl( 1 - \kappa  \NSR \norm{\Reach_\kappa(A_2, B_2)^+}\norm{\Reach_\kappa(A_2, B_2)} \bigr) - \sqrt\kappa C_1 \left(\displaystyle{\max_{i=1, \ldots, m}\abs{({\mu_\cnoise})_i}^{-1}}\right) \norm{\Reach_\kappa(A_2, B_2)^+}\norm{\Reach_\kappa(A_2, I_2)}$. We now define
			\begin{equation}
			\label{e:whatisr}
				r \Let a + \kappa\bigl(\norm{\Reach_\kappa(A_2, B_2)}U_{\max}\sqrt{{\sigma_\cnoise}} + \norm{\Reach_\kappa(A_2, I_2)} C_1/\sqrt{\kappa}\bigr) \le U_{\max}.
			\end{equation}
			By Assumption \ref{a:key}-\ref{a:key:cnoisesetandnonzeroentries}, every entry of $\bar\Cnoise$ is nonzero; we let $\bar\Cnoise^{(-1)}$ be the vector of reciprocals of each entry of $\bar\Cnoise$ (i.e., $(\bar\Cnoise^{(-1)})_i = (\bar\Cnoise_i)^{-1}$ for each $i$). We define our control policy\footnote{This controller resembles in part the Ackermann's formula in standard linear control theory \citep[p.\ 477]{ref:FraPowEma-06} employed in unconstrained deadbeat controllers.}
			\begin{equation}
			\label{e:finalcontrolpolicy}
				\uu_{\kappa t} \Let \uu_{\kappa t}\bigl(\xOrtho_{\kappa t}\bigr) \Let -\Reach_\kappa(A_2, B_2)^+ \sat_r\bigl(A_2^\kappa \xOrtho_{\kappa t}\bigr)\Sprod\bar\Cnoise^{(-1)},
			\end{equation}
			where $\sat_r$ is the function defined in \eqref{e:whatissatr}. Clearly, $\uu_{\kappa t}$ is $\sigalg_{\kappa t}$-measurable. Substituting into \eqref{e:semifinalineq} we see that
			\begin{align*}
				\EE^{\sigalg_{\kappa t}}\Bigl[\norm{\xOrtho_{\kappa(t+1)}} - \norm{\xOrtho_{\kappa t}}\Bigr] & \le -r + \kappa\bigl(\norm{\Reach_\kappa(A_2, B_2)}U_{\max}\sqrt{{\sigma_\cnoise}} + \norm{\Reach_\kappa(A_2, I_2)} C_1/\sqrt\kappa\bigr)\\
				& \qquad\qquad \text{on the set $\big\{\norm{\xOrtho_{\kappa t}} \ge r\bigr\}$}\\
				& \le -a\qquad\text{on the set $\bigl\{\norm{\xOrtho_{\kappa t}} \ge r\bigr\}$},
			\end{align*}
			where the last inequality follows from the definition of $r$ above.

			Thus, it only remains to see that the control policy defined in \eqref{e:finalcontrolpolicy} satisfies the bound $\norm{u_t} \le U_{\max}$ for each $t$. But in view of the definition of $r$ in \eqref{e:whatisr} and our policy \eqref{e:finalcontrolpolicy}, we see that $\norm{\uu_{\kappa t}} \le r \norm{\Reach_\kappa(A_2, B_2)^+} \max_{i=1, \ldots, m}\abs{(\mu_\cnoise)_i}^{-1} \le U_{\max}$, and the assertion follows immediately.
		\end{proof}

		\begin{lemma}
		\label{l:PRcondition2}
			Given the system \eqref{e:sys}, suppose that Assumption \ref{a:basic} holds, and consider the decomposition \eqref{e:Jordan}. Then there exists $M > 0$ such that
			\[
				\EE\Bigl[\abs{\norm{\xOrtho_{\kappa(t+1)}} - \norm{\xOrtho_{\kappa t}}}^4\,\Big|\,\norm{\xOrtho_0}, \ldots, \norm{\xOrtho_{\kappa t}}\Bigr] \le M\qquad \text{for all }t\in\Nz.
			\]
		\end{lemma}
		\begin{proof}
			We retain the notation $w^{(2)}_{\kappa t:\kappa(t+1)-1}$ from the proof of Lemma \ref{l:PRcondition1} and the definition of $\uu_{\kappa t}$ from \eqref{e:whatisuu}. Fix $t\in\Nz$. Observe that since $A_2$ is orthogonal, $\norm{\xOrtho_{\kappa t}} = \norm{A_2^\kappa \xOrtho_{\kappa t}}$, and therefore,
			\begin{align*}
				& \EE\Bigl[\abs{\norm{\xOrtho_{\kappa (t+1)}} - \norm{\xOrtho_{\kappa t}}}^4\,\Big|\,\bigl\{\norm{\xOrtho_{\kappa n}}\bigr\}_{n=0}^t\Bigr]\\
				& = \EE\Bigl[\Bigl|\norm{A_2^\kappa \xOrtho_{\kappa t} + \Reach_\kappa(A_2, B_2)\uu_{\kappa t} + \Reach_\kappa(A_2, I_2)w^{(2)}_{\kappa t:\kappa(t+1)-1}} - \norm{A_2^\kappa \xOrtho_{\kappa t}}\Bigr|^4\,\Big|\,\bigl\{\norm{\xOrtho_{\kappa n}}\bigr\}_{n=0}^t\Bigr]\\
				& \le \EE\Bigl[\Bigl\|A_2^\kappa \xOrtho_{\kappa t} + \Reach_\kappa(A_2, B_2)\uu_{\kappa t} + \Reach_\kappa(A_2, I_2)w^{(2)}_{\kappa t:\kappa(t+1)-1} - A_2^\kappa \xOrtho_{\kappa t}\Bigr\|^4\,\Big|\,\bigl\{\norm{\xOrtho_{\kappa n}}\bigr\}_{n=0}^t\Bigr]\\
				& = \EE\Bigl[\norm{\Reach_\kappa(A_2, B_2)\uu_{\kappa t} + \Reach_\kappa(A_2, I_2)w^{(2)}_{\kappa t:\kappa(t+1)-1}}^4\,\Big|\,\bigl\{\norm{\xOrtho_{\kappa n}}\bigr\}_{n=0}^t\Bigr].
			\end{align*}
			By Assumption \ref{a:basic}-\ref{a:basic:uconstr}, $\norm{\acn_t} \le \sqrt{m}U_{\max}\diam(\cnoiseset)$, which implies that
			\begin{align*}
				\EE& \Bigl[\norm{\Reach_\kappa(A_2, B_2)\uu_{\kappa t} + \Reach_\kappa(A_2, I_2)w^{(2)}_{\kappa t:\kappa(t+1)-1}}^4\,\Big|\,\bigl\{\norm{\xOrtho_{\kappa n}}\bigr\}_{n=0}^t\Bigr]\\
				& \le \EE\Bigl[\Bigl(\kappa \sqrt{m} U_{\max}\diam(\cnoiseset) \norm{\Reach_\kappa(A_2, B_2)} + \norm{\Reach_\kappa(A_2, I_2)}\norm{w^{(2)}_{\kappa t:\kappa(t+1)-1}}\Bigr)^4\,\Big|\,\bigl\{\norm{\xOrtho_{\kappa n}}\bigr\}_{n=0}^t\Bigr].
			\end{align*}
			Noting that $w^{(2)}_{\kappa t:\kappa(t+1)-1}$ is independent of $\norm{\xOrtho_0}, \ldots, \norm{\xOrtho_{\kappa t}}$ in view of Assumption \ref{a:basic}-\ref{a:basic:4moment}, applying Jensen's inequality to the right-hand side above yields
			\begin{align*}
				\EE & \Bigl[\Bigl(\kappa \sqrt{m} U_{\max}\diam(\cnoiseset) \norm{\Reach_\kappa(A_2, B_2)} + \norm{\Reach_\kappa(A_2, I_2)}\norm{w^{(2)}_{\kappa t:\kappa(t+1)-1}}\Bigr)^4\Bigr]\\
				& = \EE\Bigl[\Bigl(\kappa \sqrt{m} U_{\max}\diam(\cnoiseset) \norm{\Reach_\kappa(A_2, B_2)} + \kappa \norm{\Reach_\kappa(A_2, I_2)}\norm{w^{(2)}_{\kappa t}}\Bigr)^4\Bigr]\\
				& \le \kappa^4 \Bigl(\sqrt{m} U_{\max}\diam(\cnoiseset) \norm{\Reach_\kappa(A_2, B_2)} + \norm{\Reach_\kappa(A_2, I_2)}C_1\Bigr)^4.
			\end{align*}
			The assertion follows at once with $M$ equal to the right-hand side of the last inequality.
		\end{proof}

		\begin{proof}[Proof of Proposition \ref{p:main}]
			From \eqref{e:Jordan} we see that the system splits into two parts, $\xSchur$ and $\xOrtho$, with the sequence $(\xSchur_t)_{t\in\Nz}$ describing the evolution of the Schur stable component of the state, and $(\xOrtho_t)_{t\in\Nz}$ describing the evolution of the orthogonal component of the state. It is well-known that $(\xSchur_t)_{t\in\Nz}$ is mean-square bounded so long as the control is bounded, which by Assumption \ref{a:key}-\ref{a:key:cnoisesetandnonzeroentries} clearly holds (i.e., there exists $\msbound^{(1)} > 0$ such that $\EE_{\xz}\bigl[\norm{\xSchur_t}^2\bigr] \le \msbound^{(1)}$ for all $t\in\Nz$). It thus suffices to concentrate on $(\xOrtho_t)_{t\in\Nz}$. We let $\xi_t \Let \norm{\xOrtho_{\kappa t}}$ for each $t\in\Nz$. We see that:
			\begin{itemize}[label=$\circ$, leftmargin=*]
				\item the condition \eqref{e:PR1} of Proposition \ref{p:PR} holds with $J = r$, where $r$ is as defined in \eqref{e:whatisr}, by Lemma \ref{l:PRcondition1}, and
				\item the condition \eqref{e:PR2} of Proposition \ref{p:PR} holds by Lemma \ref{l:PRcondition2}.
			\end{itemize}
			Defining $J \Let \max\bigl\{\norm{\xOrtho_0}, r\bigr\}$, we see that by Proposition \ref{p:PR} there exists a $\wt{\msbound}^{(2)} = \wt{\msbound}^{(2)}(a, M, J)$ such that $\EE_{\xz}\bigl[\norm{\xOrtho_{\kappa t}}^2\bigr] \le \wt{\msbound}^{(2)}$ for all $t\in\Nz$. Since the subsampled process $(\xOrtho_{\kappa t})_{t\in\Nz}$ is mean-square bounded, and $\xOrtho$ is generated by a linear dynamical system, we conclude that there exists $\msbound^{(2)} > 0$ such that $\EE_{\xz}\bigl[\norm{\xOrtho_t}^2\bigr] \le \msbound^{(2)}$ for all $t\in\Nz$. The assertion of Proposition \ref{p:main} follows with $\msbound \Let \msbound^{(1)} + \msbound^{(2)}$ and noticing that $a$ and $J$ depend on $\xz$, $\kappa$, $\mu_\cnoise$, $\NSR$ and $C_1$.
		\end{proof}

		\begin{proof}[Proof of Proposition \ref{p:burst}]
			Let $\sigalg_t \Let \sigma(x_n,\;n = 0, 1, \ldots, t)$. Let us consider the $\kappa$-subsampled system
			\[
			    x_{\kappa (t+1)}^{(2)} = A_2^\kappa x_{\kappa t}^{(2)} +\Reach_\kappa(A_2,B_2)\cnoise_{\kappa t}u_{\kappa t:\kappa (t+1) - 1} +\Reach_\kappa(A_2, I_2) w_{\kappa t:\kappa (t+1) - 1},\qquad t\in\Nz,
			\]
			where $u_{\kappa t:\kappa(t+1)-1} \Let \begin{bmatrix} u_{\kappa t}\transp & \ldots & u_{\kappa(t+1)-1}\transp\end{bmatrix}\transp$. For this subsampled system we propose the control policy:
			\begin{equation}
			\label{e:burstpolicy}
			    u_{\kappa t:\kappa (t+1) - 1} = -\Reach_\kappa(A_2,B_2)^{+}\sat_r\bigl(A_2^\kappa \xOrtho_{\kappa t}\bigr),
			\end{equation}
			for some $r>0$ to be defined shortly. Let us verify the conditions of Proposition \ref{p:PR} for the process $\bigl(\norm{\xOrtho_{\kappa t}}\bigr)_{t\in\Nz}$ under the control policy proposed above. We see immediately that
			\begin{align*}
			\EE^{\sigalg_{\kappa t}} & \bigl[\norm{\xOrtho_{\kappa (t+1)}}-\norm{\xOrtho_{\kappa t}}\bigr]\\
				& \leq \EE^{\sigalg_{\kappa t}}\bigl[\norm{A_2^\kappa \xOrtho_{t\kappa} + \Reach_\kappa(A_2, B_2) \cnoise_{\kappa t} u_{\kappa t:\kappa (t+1) - 1}}\bigr] - \norm{\xOrtho_{\kappa t}} + \EE\bigl[\norm{\Reach_\kappa(A_2,I_2)w_{\kappa t:\kappa (t+1)-1}}\bigr]\\
				& = p\norm{A_2^\kappa \xOrtho_{\kappa t} + \Reach_\kappa(A_2,B_2) u_{\kappa t:\kappa(t+1) - 1}} + (1-p)\norm{A_2^\kappa \xOrtho_{\kappa t}}-\norm{\xOrtho_{\kappa t}}\\
				& \qquad\qquad + \EE\bigl[\norm{\Reach_\kappa(A_2, I_2) w_{\kappa t:\kappa(t+1) - 1}}\bigr]\\
				& = p\left(\norm{A_2^\kappa \xOrtho_{\kappa t} + \Reach_\kappa(A_2, B_2) u_{\kappa t:\kappa(t+1) - 1}} - \norm{\xOrtho_{\kappa t}}\right) + \sqrt{\kappa} \norm{\Reach_\kappa(A_2, I_2)} C_1\\
				& = -pr + \sqrt{\kappa} \norm{\Reach_\kappa(A_2, I_2)} C_1,
			\end{align*}
			where we have employed orthogonality of $A_2$ to arrive at the second equality above. By Assumption \ref{a:burst}-\ref{a:burst:Umaxbound} we see that there exists $a > 0$ such that $\norm{\Reach_\kappa(A_2, B_2)^+}\bigl(a + \sqrt{\kappa} C_1 \norm{\Reach_\kappa(A_2, I_2)}/p\bigr) \le U_{\max}$. Letting $r \Let a + \sqrt{\kappa} C_1 \norm{\Reach_\kappa(A_2, I_2)}/p$, we see that the condition \eqref{e:PR1} is verified with $J \Let r$. The condition \eqref{e:PR2} follows readily from Lemma \ref{l:PRcondition2}, since the elements of the control input are uniformly bounded. Letting $J \Let \max\bigl\{r, \norm{\xOrtho_0}\bigr\}$, we see that by Proposition \eqref{p:PR} there exists a constant $\msbound^{(2)} > 0$ such that $\EE_{\xz}\bigl[\norm{\xOrtho_t}^2\bigr] \le \msbound^{(2)}$ for all $t\in\Nz$. By the same argument involving the Schur stable part $\xSchur$ as in the proof of Proposition \ref{p:main}, we see that there exists a constant $\msbound' > 0$ such that $\EE_{\xz}[\norm{x_t}^2] \le \msbound'$ for all $t\in\Nz$, concluding the proof.
		\end{proof}

\end{document}